\documentclass[11pt,a4paper]{article}
\usepackage{amsmath}
\usepackage{amsthm}
\usepackage{makeidx}
\usepackage{latexsym}
\usepackage{graphicx}
\usepackage{amssymb}

\usepackage{hyperref}
\usepackage{float}
\newtheorem{theorem}{Theorem}

\newtheorem{proposition}{Proposition}

\newtheorem{remark}{Remark}
\newtheorem{definition}{Definition}
\theoremstyle{remark}
\newtheorem{example}{\textbf{Example}}
\usepackage[latin1]{inputenc}
\title{Writing $\pi$ as sum of arcotangents with linear recurrent sequences, Golden mean and Lucas numbers}
\author{Marco Abrate, Stefano Barbero, Umberto Cerruti, Nadir Murru}
\date{}

\begin{document}

\maketitle

\begin{abstract}
In this paper, we study the representation of $\pi$ as sum of arcotangents. In particular, we obtain new identities by using linear recurrent sequences. Moreover, we provide a method in order to express $\pi$ as sum of arcotangents involving the Golden mean, the Lucas numbers, and more in general any quadratic irrationality.
\end{abstract}

\section{Expressions of $\pi$ via arctangent function with linear recurrent sequences}

\indent \indent The problem of expressing $\pi$ as the sum of arctangents has been deeply studied during the years. The first expressions are due to Newton (1676), Machin (1706), Euler (1755), who expressed $\pi$ using the following identities
$$\cfrac{\pi}{2}=2\arctan \left(\cfrac{1}{2}\right)+\arctan \left(\cfrac{4}{7}\right)+\arctan \left(\cfrac{1}{8}\right)$$
$$\cfrac{\pi}{4}=\arctan \left(\cfrac{1}{2}\right)+\arctan\left(\cfrac{1}{3}\right)$$
$$\cfrac{\pi}{4}=5\arctan\left(\cfrac{1}{7}\right)+2\arctan \left(\cfrac{3}{79}\right),$$
respectively (see, e.g., \cite{Tw} and \cite{W}). Many other identities and methods to express and calculate $\pi$ involving the arctangent function have been developed. Some recent results are obtained in \cite{Lih} and \cite{Cal}.\\
\indent In this section, we find a method to generate new expressions of $\pi$ in terms of sum of arctangents, mainly using the properties of linear recurrent sequences. For the sake of simplicity, we will use the following notation:
$$A(x)=\arctan(x).$$
It is well--known that for $x,y \geq0$, if $\displaystyle{y\not=\cfrac{1}{x}}$
$$A(x)+A(y) =
\begin{cases} 
A(x\odot y) &\qquad \text{if} \qquad xy<1, \cr 
A(x\odot y) + \text{sign}(x)\pi&\qquad \text{if} \qquad xy>1,
\end{cases}
$$
where
$$x\odot y=\cfrac{x+y}{1-xy}.$$
Let us denote by $x^{\odot n}$ the $n$--th power of $x$ with respect to the product $\odot$.
\begin{remark}
The product $\odot$ is associative, commutative and 0 is the identity.
\end{remark}

\begin{definition}
We denote by $a=(a_n)_{n=0}^{+\infty}=\mathcal W(\alpha,\beta,p ,q)$ the linear recurrent sequence of order 2 with characteristic polynomial $t^2-pt+q$ and initial conditions $\alpha$ and $\beta$, i.e.,
$$\begin{cases} a_0=\alpha \cr a_1=\beta \cr a_n=pa_{n-1}-qa_{n-2}\quad \forall n\geq2\ .  \end{cases}$$
\end{definition}
\begin{theorem} \label{main}
Given $n\in\mathbb N$ and $x\in\mathbb R$, with $x\not=\pm1$, we have
$$\left( \cfrac{1}{x} \right)^{\odot n}=\cfrac{v_n(x)}{u_n(x)},\quad \forall n\geq1$$
where 
\begin{equation}\label{eq:u-v}
(u_n(x))_{n=0}^\infty=\mathcal W(1,x,2x,1+x^2),\quad (v_n(x))_{n=0}^\infty=\mathcal W(0,1,2x,1+x^2).
\end{equation}
\end{theorem}
\begin{proof}
The matrix
$$M=\begin{pmatrix} x & 1 \cr -1 & x  \end{pmatrix}$$
has characteristic polynomial $t^2-2xt+x^2+1$. Consequently, it is immediate to see that
$$M^n=\begin{pmatrix} u_n(x) & v_n(x) \cr -v_n(x) & u_n(x)  \end{pmatrix}.$$
Using the matrix $M$ we can observe that
$$\begin{pmatrix} u_{n-1}(x) & v_{n-1}(x) \cr -v_{n-1}(x) & u_{n-1}(x) \end{pmatrix}\begin{pmatrix} x & 1 \cr -1 & x  \end{pmatrix}=\begin{pmatrix} u_n(x) & v_n(x) \cr -v_n(x) & u_n(x)  \end{pmatrix},$$
i.e., 
$$\begin{cases} u_n(x)=xu_{n-1}(x)-v_{n-1}(x) \cr v_n(x)=u_{n-1}(x)+xv_{n-1}(x) \end{cases}, \quad \forall n\geq1.$$
Now, we prove the theorem by induction. It is straightforward to check that
$$\cfrac{1}{x}=\cfrac{v_1(x)}{u_1(x)},\quad \left( \cfrac{1}{x} \right)^{\odot 2}=\cfrac{\frac{1}{x}+\frac{1}{x}}{1-\frac{1}{x^2}}=\cfrac{2x}{x^2-1}=\cfrac{v_2(x)}{u_2(x)}.$$
Moreover, let us suppose
$$\left( \cfrac{1}{x} \right)^{\odot (n-1)}=\cfrac{v_{n-1}(x)}{u_{n-1}(x)}$$
for a given integer $n\geq 1$, then
$$\left( \cfrac{1}{x} \right)^{\odot n}=\cfrac{1}{x}\odot\left( \cfrac{1}{x} \right)^{\odot (n-1)}=\cfrac{1}{x}\odot\cfrac{v_{n-1}(x)}{u_{n-1}(x)}=\cfrac{u_{n-1}(x)+xv_{n-1}(x)}{xu_{n-1}(x)-v_{n-1}(x)}=\cfrac{v_n(x)}{u_n(x)}.$$
\end{proof}

\begin{theorem} \label{th:x^n}
Given $n\in\mathbb N$ and $x\in\mathbb R$, with $x\not=\pm1$, we have
$$x^{\odot n}=(-1)^{n+1}\left(\cfrac{v_n(x)}{u_n(x)}\right)^{(-1)^n},\quad \forall n\geq1$$
where $u_n(x)$ and $v_n(x)$ are given by Eq.\eqref{eq:u-v}. 
\end{theorem}
\begin{proof}
By using the same arguments of Theorem \ref{main}, we can write
$$
x = \cfrac{u_1(x)}{v_1(x)}\qquad \text{and} \qquad x^{\odot 2}=\cfrac{2x}{1-\frac{1}{x^2}}=-\cfrac{v_2(x)}{u_2(x)}.
$$
Let us suppose by induction that $x^{\odot (n - 1)}=(-1)^{n}\left(\cfrac{v_{n-1}(x)}{u_{n-1}(x)}\right)^{(-1)^{n-1}}$, then if $n$ is even	
$$x^{\odot n}=\cfrac{x - \cfrac{v_{n-1}(x)}{u_{n-1}(x)}}{1+x\cfrac{v_{n-1}(x)}{u_{n-1}(x)}} = \cfrac{xu_{n-1}(x)-v_{n-1}(x)}{u_{n-1}(x)+xv_{n-1}(x)} = \cfrac{u_{n}(x)}{v_{n}(x)},$$
if $n$ is odd
$$x^{\odot n}=\cfrac{x + \cfrac{u_{n-1}(x)}{v_{n-1}(x)}}{1-x\cfrac{u_{n-1}(x)}{v_{n-1}(x)}} = \cfrac{xv_{n-1}(x)+u_{n-1}(x)}{v_{n-1}(x)-xu_{n-1}(x)} = -\cfrac{v_{n}(x)}{u_{n}(x)}.$$
\end{proof}

Let us highlight the matrix representation of the sequences $(u_n)_{n=0}^\infty$ and $(v_n)_{n=0}^\infty$ used in the previous theorem. Given the matrix 
$$M=\begin{pmatrix} x & 1 \cr -1 & x  \end{pmatrix}$$
we have
$$M^n=\begin{pmatrix} u_n(x) & v_n(x) \cr -v_n(x) & u_n(x)  \end{pmatrix}$$
$$ M^n\begin{pmatrix} v_m(x) \cr u_m(x) \end{pmatrix}=\begin{pmatrix} v_{n+m}(x) \cr u_{n+m}(x) \end{pmatrix}$$
\indent The sequences $(u_n)_{n=0}^\infty$ and $(v_n)_{n=0}^\infty$ are particular cases of the R\'edei polynomials $N_n(d,z)$ and $D_n(d,z)$,  introduced by R\'edei \cite{Redei} from the expansion of $(z+\sqrt{d})^n=N_n(d,z)+D_n(d,z)\sqrt{d}$. The rational functions $\cfrac{N_n(d,z)}{D_n(d,z)}$ have many interesting properties, e.g. , they are permutations of finite fields, as described in the book of Lidl  \cite{Lidl}. In \cite{bcm}, the authors showed that R\'edei polynomials are linear recurrent sequences of degree 2:
$$(N_n(d,z))_{n=0}^\infty=\mathcal W(1,z,2z,z^2-d),\quad (D_n(d,z))_{n=0}^\infty=\mathcal W(0,1,2z,z^2-d).$$
Thus, we can observe that
$$u_n(x)=N_n(-1,x), \quad v_n(x)=D_n(-1,x),\quad \forall n\geq0.$$
Moreover, a closed expression of R\'edei polynomials is well--known (see, e.g., \cite{bcm}). In this way, we can derive a closed expression for the sequences $(u_n)_{n=0}^\infty$ and $(v_n)_{n=0}^\infty$:
\begin{equation}\label{eq:redei}
\begin{cases} 
\displaystyle{u_n(x)=\sum_{k=0}^{[n/2]}\binom{n}{2k}(-1)^kx^{n-2k}} \cr 
\displaystyle{v_n(x)=\sum_{k=0}^{[n/2]}\binom{n}{2k+1}(-1)^kx^{n-2k-1}}  
\end{cases}.
\end{equation}

\indent Rational powers with respect to the product $\odot$ can also be considered by defining the $n$--th root as usual by
\begin{equation}\label{eq:root}
z = x^{\odot \cfrac{1}{n}} \qquad \text{iff} \qquad z^{\odot n} = x.
\end{equation}

Moreover, by means of Theorem \ref{th:x^n}, we have that Eqs. \eqref{eq:root} are equivalent to
$$
x=(-1)^{n+1}\left(\cfrac{v_n(z)}{u_n(z)}\right)^{(-1)^n},
$$
i.e., by Eqs. \eqref{eq:redei}, the $n$--th root of $x$ with respect to the product $\odot$ is a root of the polynomial
$$
P_n(z) = \sum_{k=0}^{n} \binom{n}{k} (-1)^{\left\lfloor\frac{k+1}{2}\right\rfloor}x^{\frac{1+(-1)^{k+1}}{2}}z^k.
$$

Let us consider the equation
\begin{equation} \label{aeq} nA\left(\cfrac{1}{x}\right)+A\left(\cfrac{1}{y}\right)=\cfrac{\pi}{4}, \end{equation}
we want to solve it when $n$ and $x$ are integer values. We point out that Eq. \eqref{aeq} is equivalent to
\begin{equation} \label{eq} \left(  \cfrac{1}{x} \right)^{\odot n}\odot \cfrac{1}{y}=1 \end{equation}

By Theorem \ref{main} we have
$$\left( \cfrac{1}{x} \right)^{\odot n}\odot \cfrac{1}{y}=\cfrac{v_n(x)}{u_n(x)}\odot\cfrac{1}{y}=\cfrac{u_n(x)+v_n(x)y}{-v_n(x)+u_n(x)y}.$$
Thus 
$$y=\cfrac{u_n(x)+v_n(x)}{u_n(x)-v_n(x)}$$
solves Eq. \eqref{eq}, i.e., 
$$\left( \cfrac{1}{x} \right)^{\odot n}\odot\cfrac{u_n(x)+v_n(x)}{u_n(x)-v_n(x)}=1,\quad \forall x\in\mathbb Z$$
and consequently we can solve Eq. \eqref{aeq}, i.e.,
\begin{equation} \label{maineq} nA\left(\cfrac{1}{x}\right)+A\left(\cfrac{u_n(x)-v_n(x)}{u_n(x)+v_n(x)}\right)=\cfrac{\pi}{4}+k(n,x)\pi,\quad \forall x\in\mathbb Z, \end{equation}
where $k$ is a certain integer number depending on $n$ and $x$. Precisely, we have
\begin{equation} \label{k} k(n,x)=\text{sign}\left(nA\left(\cfrac{1}{x}\right)-\cfrac{\pi}{4}\right)\left(\lfloor T\rfloor +\chi_{\left(\frac{1}{2},1\right)}\left(\{T\}\right)\right), \end{equation}
where $\chi_{\left(\frac{1}{2},1\right)}$ is the characteristic function of the set $\left(\frac{1}{2},1\right)$ and 
$$T=\cfrac{\left\lvert \cfrac{\pi}{4}-nA\left(\cfrac{1}{x}\right) \right\rvert}{\pi}.$$
In order to obtain Eq. \eqref{k}, we can rewrite Eq. \eqref{maineq} as
$$A\left(\cfrac{u_n(x)-v_n(x)}{u_n(x)+v_n(x)}\right)=\cfrac{\pi}{4}-nA\left(\cfrac{1}{x}\right)+k(n,x)\pi.$$
Let us consider the case in which the first member lies in the interval  $\displaystyle{\left(-\frac{\pi}{2},\frac{\pi}{2}\right)}$. If $\displaystyle{\cfrac{\pi}{4}-nA\left(\cfrac{1}{x}\right)\geq0}$, then $k(n,x)$ must be negative so that $\cfrac{\pi}{4}-nA\left(\cfrac{1}{x}\right)+k(n,x)\pi$ lies in the correct interval. Since
$$\cfrac{\pi}{4}-nA\left(\cfrac{1}{x}\right)=\pi\left(\lfloor T \rfloor + \{T\}\right),$$
it follows that if $\displaystyle{0\leq \{T\} \leq \frac{1}{2}}$, then $\displaystyle{0\leq \pi\cdot \{T\} \leq \frac{\pi}{2}}$ and consequently $k=-\lfloor T \rfloor$. Conversely, if $\displaystyle{\frac{1}{2} < \{T\} < 1}$, then $\displaystyle{\frac{\pi}{2} < \pi\cdot\{T\} < \pi}$ and, observing that
$$\cfrac{\pi}{4}-nA\left(\cfrac{1}{x}\right)=\pi\left(\lfloor T\rfloor+1\right)+ \pi\left(\{T\}-1\right),$$
we obtain $\displaystyle{-\frac{\pi}{2} < \pi(\{T\}-1)<0}$, that is $k(n,x)=-(\lfloor T \rfloor +1)$. 

\noindent Similar considerations apply to $\cfrac{\pi}{4}-nA\left(\cfrac{1}{x}\right)<0$, obtaining Eq. \eqref{k}.

\begin{proposition}
The sequences $(u_n(x)+v_n(x))_{n=0}^\infty$ and $(u_n(x)-v_n(x))_{n=0}^\infty$ are linear recurrent sequences of order 2 and precisely
$$(u_n(x)+v_n(x))_{n=0}^\infty=\mathcal W(1,x+1,2x,1+x^2),\quad (u_n(x)-v_n(x))_{n=0}^\infty=\mathcal W(1,x-1,2x,1+x^2)$$
\end{proposition}
\begin{proof}
It immediately follows from the definition of the sequences $(u_n)_{n=0}^\infty$ and $(v_n)_{n=0}^\infty$.
\end{proof}
Eq. \eqref{maineq} provides infinitely many identities that express $\pi$ as sum of arctangents.
\begin{example}
Taking $n=7$ and $x=3$ in Eq. \eqref{maineq} we have
$$7A\left(\cfrac{1}{3}\right)+A\left(\cfrac{u_7(3)-v_7(3)}{u_7(3)+v_7(3)}\right)=\cfrac{\pi}{4},$$
i.e.,
$$7\arctan\left(\cfrac{1}{3}\right)-\arctan\left(\cfrac{278}{29}\right)=\cfrac{\pi}{4}.$$
For $n=8$ and $x=3$, we have
$$8\arctan\left(\cfrac{1}{3}\right)+\arctan\left(\cfrac{863}{191}\right)=\cfrac{\pi}{4}+\pi.$$
For $n=5$ and $x=2$, we have
$$5\arctan\left(\cfrac{1}{2}\right)-\arctan\left(\cfrac{79}{3}\right)=\cfrac{\pi}{4}.$$
For $n=2$ and $x=7$, we have
$$2\arctan\left(\cfrac{1}{7}\right)+\arctan\left(\cfrac{17}{31}\right)=\cfrac{\pi}{4}.$$
\end{example}

\section{Golden mean and $\pi$}
\indent \indent In Mathematics the most famous numbers are $\pi$ and the Golden mean. Thus, it is very interesting to find identities involving these special numbers. In particular, many expressions for $\pi$ in terms of the Golden mean have been found. For example, using the Machin formula of $\pi$ via arctangents,  the following equalities arise
$$\cfrac{\pi}{4}=\arctan\left( \cfrac{1}{\phi} \right)+\arctan \left(\cfrac{1}{\phi^3} \right)$$
$$\cfrac{\pi}{4}=2\arctan\left( \cfrac{1}{\phi^2} \right)+\arctan \left(\cfrac{1}{\phi^6} \right)$$
$$\cfrac{\pi}{4}=3\arctan\left( \cfrac{1}{\phi^3} \right)+\arctan \left(\cfrac{1}{\phi^5} \right)$$
$$\pi=12\arctan\left( \cfrac{1}{\phi^3} \right)+4\arctan \left(\cfrac{1}{\phi^5} \right),$$
see \cite{C1}, \cite{C2}, \cite{C3}. Moreover, in \cite{Luca}, the authors found all possible relations of the form
$$\cfrac{\pi}{4}=a\arctan(\phi^k)+b\arctan(\phi^l),$$
where $a,b$ are rational numbers and $k,l$ integers.\\
\indent In this section, we find new expressions of $\pi$ as sum of arctangents involving $\phi$. When $n=2$, from Eq. \eqref{eq} we find
\begin{equation}  \label{y} y=\cfrac{x^2+2x-1}{x^2-2x-1}. \end{equation}
It is well--known that the minimal polynomial of $\phi^m$ is
$$f_m(t)=t^2-L_mt+(-1)^m,$$
where $(L_m)_{m=0}^\infty=\mathcal W(2,1,1,-1)$ is the sequence of Lucas numbers (A000032 in OEIS \cite{oeis}). If we set $x=\phi^m$ in \eqref{y}, then it is equivalent to replace $x^2+2x-1$ and $x^2-2x-1$ with
$$x^2+2x-1 \pmod{f_m(x)},\quad x^2-2x-1 \pmod{f_m(x)},$$
respectively. When $m$ is odd, dividing by $x^2-L_mx-1$, we obtain
$$y=\cfrac{(L_m+2)x}{(L_m-2)x}=\cfrac{L_m+2}{L_m-2}$$
and when $m$ is even, we have
$$y=\cfrac{-2+(2+L_m)x}{-2+(-2+L_m)x},$$
and therefore
$$y=\cfrac{-2+(2+L_m)\phi^m}{-2+(-2+L_m)\phi^m}.$$
We find the following identities
\begin{equation} \label{phi1} \cfrac{\pi}{4}=2\arctan\left( \cfrac{1}{\phi^{2k+1}} \right)+\arctan\left( \cfrac{L_{2k+1}-2}{L_{2k+1}+2} \right) \end{equation}
$$\cfrac{\pi}{4}=2\arctan\left( \cfrac{1}{\phi^{2k}} \right)+\arctan\left( \cfrac{-2+(L_{2k}-2)\phi^{2k}}{-2+(L_{2k}+2)\phi^{2k}} \right).$$
The above procedure can be reproduced for any root $\alpha$  of a polynomial $x^2-hx+k$, finding expression of $\pi$ as the sum of arctangents involving quadratic irrationalities.
\begin{example}
Let us express $\pi$ in terms of $\sqrt{2}$. Its minimal polynomial is $x^2-2$ and 
$$x^2+2x-1 \pmod{x^2-2}=1+2x,\quad x^2-2x-1 \pmod{x^2-2}=1-2x.$$
We have
$$\cfrac{\pi}{4}=2\arctan\left(\cfrac{1}{\sqrt{2}}\right)+\arctan\left(\cfrac{1-2\sqrt{2}}{1+2\sqrt{2}}\right).$$
In general, if $k$ is odd the minimal polynomial of $\sqrt{2^k}$ is $x^2-2^k$ and
$$x^2+2x-1 \pmod{x^2-2^k}=2^k-1+2x,\quad x^2-2x-1 \pmod{x^2-2^k}=2^k-1-2x.$$
We have the following identity
$$\cfrac{\pi}{4}=2\arctan\left(\cfrac{1}{\sqrt{2^k}}\right)+\arctan\left(\cfrac{2^k-1-2^{\frac{k}{2}+1}}{2^k-1+2^{\frac{k}{2}+1}}\right).$$
\end{example} 
\begin{example}
Let us consider $\alpha=\cfrac{1}{2}(5+\sqrt{29})$. The minimal polynomial of $\alpha^3$ is $x^2-140x-1$ and
$$x^2+2x-1 \pmod{x^2-140x-1}=142x,\quad x^2-2x-1 \pmod{x^2-140x-1}=138x.$$
Thus, we have
$$\cfrac{\pi}{4}=2\arctan\left( \cfrac{8}{(5+\sqrt{29})^3} \right)+\arctan\left( \cfrac{69}{71} \right).$$
\end{example}
We can find different identities involving $\pi$ and the Golden mean considering the equation
\begin{equation}\label{eq:piphi}
x^{\odot\frac{1}{2}}\odot y = 1.
\end{equation}
\begin{proposition}
For any real number $x$, the following equalities hold
\begin{equation}\label{eq:pi/2}
2A(-x\pm\sqrt{1+x^2})+A(x)=\pm\cfrac{\pi}{2}.
\end{equation}
\end{proposition}
\begin{proof}
\noindent By Theorem \ref{th:x^n} we know that the roots of the polynomial $P_2(z) = xz^2+2z-x$ are the values of $x^{\odot\frac{1}{2}}$. Hence, from Eq. \eqref{eq:piphi} we obtain 
\begin{equation}\label{eq:zy}
z_{i}\odot y = 1, \quad i=1,2 ,
\end{equation}
\noindent where $$\displaystyle{z_1 = \cfrac{-1 + \sqrt{1+x^2}}{x}} \quad \text{and} \quad \displaystyle{z_2 = \cfrac{-1 - \sqrt{1+x^2}}{x}}.$$

\noindent Finally, solving Eq. \eqref{eq:piphi} with respect to $y$ we get 
$$y_1=-x+\sqrt{1+x^2} \qquad \text{or} \qquad y_2=-x-\sqrt{1+x^2}.$$

\noindent It should be noted that if $x$ is positive then $y_2<0$ and $z_2\cdot y_2 > 1$ so that 
$$
\cfrac{1}{2}A(x) + A(y_2) = A\left(x^{\odot \frac{1}{2}} + y_2\right) - \cfrac{\pi}{2},
$$
\noindent similar reasoning can be applied if $x$ is negative.

Now, substituting in Eqs. \eqref{eq:zy} we have
$$\cfrac{1}{2}A(x)+A(-x\pm\sqrt{1+x^2})=\pm\cfrac{\pi}{4},$$
or equivalently
$$
2A(-x\pm\sqrt{1+x^2})+A(x)=\pm\cfrac{\pi}{2}.
$$
\end{proof}

Eqs. \eqref{eq:pi/2} yield to other interesting formulas involving $\pi$, $\phi$ and Lucas numbers. To show this, we need some identities about Lucas numbers, Fibonacci numbers and the Golden mean:
$$\phi^m=\cfrac{L_m+F_m\sqrt{5}}{2},\quad L_m^2-5F_m^2=4(-1)^m,$$
see, e.g., \cite{Rab}.
Considering $m$ odd, if we set
$$x=\cfrac{L_m}{2}$$
it follows 
\begin{equation}\label{eq:y-}
-x-\sqrt{1+x^2}=\cfrac{-L_m-\sqrt{4+L_m^2}}{2}=\cfrac{-L_m-F_m\sqrt{5}}{2}=- \phi^m.
\end{equation}

Thus, substituting Eq. \eqref{eq:y-} into Eqs. \eqref{eq:pi/2} we find the formula 
\begin{equation}
-\cfrac{\pi}{2}=\arctan\left(\cfrac{L_{2k+1}}{2}\right)-2\arctan\left(\phi^{2k+1}\right).
\end{equation}

On the other hand, if we consider $y = -x + \sqrt{1 + x^2}$ we have
\begin{equation}\label{eq:y+}
-x+\sqrt{1+x^2}=\cfrac{-L_m+\sqrt{4+L_m^2}}{2}=\cfrac{-L_m+F_m\sqrt{5}}{2}.
\end{equation}
Moreover,
$$\phi^m\cdot \cfrac{-L_m+F_m\sqrt{5}}{2}=\cfrac{-L_m^2+5F_m^2}{4}=1,$$

\noindent and substituting in Eqs. \eqref{eq:pi/2} another interesting formula arises
\begin{equation}
\cfrac{\pi}{2}=\arctan\left(\cfrac{L_{2k+1}}{2}\right)+2\arctan\left(\cfrac{1}{\phi^{2k+1}}\right).
\end{equation}
\noindent Furthermore, by Eq. \eqref{phi1} we obtain an identity that only involves the Lucas numbers
\begin{equation} \label{onlyL} \cfrac{\pi}{4}=\arctan\left(\cfrac{L_{2k+1}}{2} \right)-\arctan\left(\cfrac{L_{2k+1}-2}{L_{2k+1}+2} \right). \end{equation}
The previous identity  corresponds to a special case of the following proposition.
\begin{proposition}
Let $f,g$ be real functions. If
$$g(x)=\cfrac{f(x)-1}{f(x)+1},$$
then
\begin{equation} \label{afg} A(f(x))-A(g(x))=\cfrac{\pi}{4}+k\pi, \end{equation}
for some integer $k$.
\end{proposition}
\begin{proof}
We use the product $\odot$ for solving $A(f(x))-A(g(x))=\cfrac{\pi}{4}$. We have
$$A\left( \cfrac{f(x)-g(x)}{1+f(x)g(x)} \right)=\cfrac{\pi}{4}$$
and
$$\cfrac{f(x)-g(x)}{1+f(x)g(x)}=1$$
from which
$$g(x)=\cfrac{f(x)-1}{f(x)+1}.$$
\end{proof}
\begin{remark}
Eq. \eqref{afg} has been found by means of only elementary algebraic considerations. The same result could be derived from analysis. Observe that given the functions $f$ and $g$ satisfying the hypothesis of the previous proposition, then $(\arctan f(x))'=(\arctan g(x))'$.
\end{remark}

When $f(x)$ and $g(x)$ are specified in Eq. \eqref{afg}, the value of $k$ can be retrieved as in Eq. \eqref{k}  with analogous considerations.\\
\indent The previous proposition allows to determine new beautiful identities. For example, the function $f(x)=\cfrac{ax}{b}$ determines the function $g(x)=\cfrac{ax-b}{ax+b}$ and
$$A\left(\cfrac{ax}{b}\right)-A\left( \cfrac{ax-b}{ax+b} \right)=\cfrac{\pi}{4}+k\pi.$$
For $a=1$ and $b=2$, we obtain the following interesting formulas 
\begin{equation} \label{xx} \cfrac{\pi}{4}=\arctan\left( \cfrac{x}{2} \right)-\arctan\left(\cfrac{x-2}{x+2}\right),\end{equation}
which holds for any real number $x>-2$ and
\begin{equation} \label{xx2} -\cfrac{3\pi}{4}=\arctan\left( \cfrac{x}{2} \right)-\arctan\left(\cfrac{x-2}{x+2}\right),\end{equation}
valid for any real number $x<-2$.
Eqs. \eqref{xx} and \eqref{xx2} provide infinitely many interesting identities, like Eq. \eqref{onlyL} and, e.g.,  the following ones
$$\cfrac{\pi}{4}=\arctan\left( \cfrac{\phi}{2} \right)-\arctan\left(\cfrac{\phi-2}{\phi+2}\right)$$
$$\cfrac{\pi}{4}=\arctan\left( \cfrac{F_m}{2} \right)-\arctan\left(\cfrac{F_m-2}{F_m+2}\right)$$
$$\cfrac{\pi}{4}=\arctan\left( \cfrac{\sqrt{2}}{2} \right)-\arctan\left(\cfrac{\sqrt{2}-2}{\sqrt{2}+2}\right).$$


\begin{thebibliography}{0}

\bibitem{bcm} S. Barbero, U. Cerruti and N. Murru, Solving the Pell equation via Rédei rational functions, {\it The Fibonacci Quarterly} \textbf{48} (2010) 348--357.

\bibitem{Cal} J. S. Calcut, Gaussian integers and arctangent identities for $\pi$, {\it The American Mathematical Monthly} \textbf{116}(6) (2009) 515--530.

\bibitem{C1} H. C. Chan and S. Ebbing, $\pi$ in terms of $\phi$: Some recent developments, {\it Proc. of the Twelfth International Conference in Fibonacci Numbers and Their Applications} (San Francisco State University, 2006).

\bibitem{C2} H. C. Chan, $\pi$ in terms of $\phi$, {\it The Fibonacci Quarterly} \textbf{44}(2) (2006) 141--145.

\bibitem{C3} H. C. Chan, Machin--type formulas expressing $\pi$ in terms of $\phi$, {\it The Fibonacci Quarterly} \textbf{46/48}(1), (2008/2009) 32--37.

\bibitem{Lih} H. Chien--Lih, Some observations on the method of arctangents for the calculation of $\pi$, {\it The Mathematical Gazette} \textbf{88}(512) (2004) 270--278.

\bibitem{Lidl} R. Lidl, G. L. Mullen and G. Turnwald, Dickson polynomials, {\it Pitman Monogr. Surveys Pure appl. Math.} \textbf{65} (Longman, 1993).

\bibitem{Luca} F. Luca and P. Stanica, On Machin's formula with powers of the Golden section, {\it International Journal of Number Theory} \textbf{05}(973) (2009).

\bibitem{Rab} S. Rabinowitz, Algorithmic manipulation of Fibonacci identities, {\it Applications of Fibonacci Numbers} \textbf{6} (1996) 389--408.

\bibitem{Redei} L. R\'{e}dei, Uber eindeuting umkehrbare polynome in endlichen korpen, {\it Acta Sci. Math. (Szeged)} \textbf{11} (1946) 85--92.

\bibitem{oeis} N. J. A. Sloane, The On--Line Encyclopedia of Integer Sequences, Published electronically
at http://www.research.att.com/njas/sequences (2010).

\bibitem{Tw} I. Tweddle, John Machin and Robert Simson on inverse--tangent series for $\pi$, {\it Arch. Hist. Exact Sci.} \textbf{42} (1991) 1--14.

\bibitem{W} J. W. Wrench, The evolution of extended decimal approximations of $\pi$, {\it Math. Teacher} (1960) 644--650.

\end{thebibliography}
\end{document}